\documentclass[11pt,a4paper]{article}

\usepackage{amsmath,amsthm,amsfonts}
\usepackage{graphicx}
\usepackage{color}
\usepackage{paralist}

\usepackage{hyperref}
\hypersetup{%
plainpages=false,
colorlinks,urlcolor=blue,linkcolor=blue,citecolor=blue,
pdftitle={},
pdfauthor={James Cruise and Stan Zachary},
pdfsubject={Optimal scheduling of energy storage resources},
pdfpagemode={UseOutlines},
pdfpagelayout={OneColumn},
pdfstartview={Dubhe}
}

\newtheorem{theorem}{Theorem}
\theoremstyle{remark}
\newtheorem{remark}{Remark}
\newtheorem{example}{Example}

\newcommand{\E}{\mathbf{E}}

\newcommand{\Rp}{\mathbb{R}_+}

\renewcommand{\S}{S}
\newcommand{\Sne}{S_{ne}}
\newcommand{\Se}{S_e}

\newcommand{\rt}{r}
\newcommand{\Pb}{\ensuremath{\mathbf{P}}}

\newcommand{\lole}{\ensuremath{\mathrm{LOLE}}}
\newcommand{\eeu}{\ensuremath{\mathrm{EEU}}}

\newcommand{\efc}{\ensuremath{\mathit{efc}}}

\newcommand{\ltf}{LRTF}

\parskip\smallskipamount
\parskip \smallskipamount

\title{Optimal scheduling of energy storage resources}
\author{James Cruise\footnote{Heriot-Watt University.  Research
    supported by EPSRC grant EP/I017054/1}
  \ and Stan Zachary\footnote{University of Edinburgh.  Research
    supported by EPSRC grant EP/I017054/1 and by a grant from the
    Simons Foundation}}
\date{\today}

\begin{document}

\maketitle

\begin{abstract}
  It is likely that electricity storage will play a significant role
  in the balancing of future energy systems.  A major challenge is
  then that of how to assess the contribution of storage to capacity
  adequacy, i.e.\ to the ability of such systems to meet demand.  This
  requires an understanding of how to optimally schedule multiple
  storage facilities.  The present paper studies this problem in the
  cases where the objective is the minimisation of expected energy
  unserved (EEU) and also when it is a form of weighted EEU in which
  the unit cost of unserved energy is higher at higher levels of unmet
  demand.  We also study how the contributions of individual stores
  may be identified for the purposes of their inclusion in electricity
  capacity markets.
\end{abstract}

\section{Introduction}
\label{sec:introduction}

In electricity systems there is a need to keep supply and demand
carefully balanced at all times.  However, the increasing penetration
of renewable generation means that future systems are likely to
characterised by much greater variability and uncertainty on the
supply side, while patterns of demand will continue to vary
considerably according to the time of day.  It is thus likely that
electricity storage will play a significant role in balancing such
future systems---see~\cite{Strbacetal} and, for a recent analysis of
the impact of storage in European markets, see~\cite{NSPN}.  The
problem of the optimal operation and control of storage may be viewed
in several ways.  For the storage operator much of the value of
storage may be realised in price arbitrage, i.e.\ in buying
electricity when it is cheap and selling it when it is expensive
(see~\cite{CFGZ,CFZ,SDJW} and the references therein) and in the
provision of buffering and ancillary services
(see~\cite{BGK,CZ,DEKM,GTL,PADS}).  See also the literature review on
the management of energy storage systems by Weitzel and Glock
\cite{WEITZEL2018582}.  However, for society and for the electricity
system operator, a major concern is that of \emph{capacity adequacy},
i.e.\ of ensuring that there is sufficient available supply to be able
to meet demand in all but the most exceptional circumstances.  Here,
as indicated above, storage facilities may be used to cover periods of
what would otherwise be shortfall in supply, and it is necessary to
assess the contributions of such storage facilities and to value them
individually for inclusion in capacity markets.  For example, in Great
Britain such valuation of individual storage facilities is now
considered by the system operator, National Grid plc, in its annual
electricity capacity reports, see~\cite{NGECR18}.

It is typically the case that continuous periods of what would
otherwise be shortfall and which are to be covered by storage are well
separated in time, so that stores may be fully recharged between such
periods.  (In Great Britain and in many other countries, for example,
the system is presently only significantly at risk during a single
evening peak period, and there is ample time for overnight
recharging.)  Thus in analysing the contributions of given storage
facilities, it is typically sufficient to consider single periods of
possible shortfall during which no recharging of stores may take
place.  However, within any such analysis there now arises the problem
of the optimal scheduling of multiple stores, with the objective of
minimising, for example, unserved energy.  This problem arises because
individual stores are subject to both capacity and rate constraints,
and there is a danger that, if the rates at which these stores are
used are not well coordinated over time, one may arrive at some time at
which there is sufficient remaining energy in the stores to meet
future needs, but that this energy is contained in too few stores
which cannot between them supply that energy at the required rates.

As discussed above, it is only recently that the problem of optimal
scheduling of multiple stores has become important in the analysis of
electricity capacity adequacy, and perhaps for this reason it does not
appear to have received much attention in the literature. There have
been a small number of studies using either dynamic programming or
Monte Carlo simulations, see \cite{KHAN201839, Sioshansi2014,
  ZHOU201512}.  However, we take a more analytic and direct approach
to the problem, as does recent work by Evans \emph{et al.}~\cite{EAT},
which we discuss in Section~\ref{sec:stor-demand-prof}.

The present paper further considers this optimal scheduling problem.
In Section~\ref{sec:model} we present the underlying model, which is
that of a nonnegative \emph{demand process} over some given continuous
period of time, together with \emph{set of stores}, each with given
rate and capacity constraints, which may not be recharged during that
period of time and which are to be used to serve that demand process
as far as possible.  The demand process might be a process of
remaining energy demand after initial demand had been met as far as
possible from sources such as generation.  This process may be
sufficiently well known in advance as to be capable of being modelled
as deterministic.  Alternatively it may be necessary to treat it as
stochastic.  We consider also possible objective functions for the
scheduling of the stores.  In Section~\ref{sec:stor-demand-prof} we
give a simple necessary and sufficient condition for a given demand
process to be completely satisfiable by a given set of stores as
described above.  In Section~\ref{sec:minimisation-eeu} we study the
problem of scheduling stores for the minimisation of \emph{expected
  energy unserved} (EEU).  It turns out that in this case the optimal
decision at any time may be made independently of the demand process
subsequent to that time, so that it is possible to construct a
solution which is optimal in both a deterministic and a stochastic
environment.  In Section~\ref{sec:contr-indiv-stor}, we consider the
contribution to individual stores to the minimisation of EEU.  This is
particularly important in the implementation of electricity capacity
markets, where capacity-providing resources of different types are in
competition with each other---see, e.g.~\cite{NGECR18}.  In
Section~\ref{sec:minim-weight-eeu-1} we study the important variation
of the scheduling problem in which the cost of unserved energy is
higher at higher levels of unmet demand; here in a stochastic
environment the solution of the problem is much more difficult.
Finally, in Section~\ref{sec:examples-1} we give a number of
illustrative examples.

\section{Model}
\label{sec:model}

In this section we give our model for a demand process over some given
period of time, together with a set of stores whose energy is to be
used to meet that demand as far as possible in accordance with some
criterion.  (As previously remarked, the demand process corresponds to
that energy demand which requires specifically to be met from
storage.)  Thus the model is defined by the following two components:
\begin{compactenum}[(a)]%
\item a nonnegative \emph{demand process} $(d(t),\,t\in[0,T])$ defined
  over some continuous period of time~$[0,T]$; this demand process may
  be deterministic or stochastic;
\item a set~$\S$ of \emph{stores}, where each store~$i\in\S$ is
  characterised by its rate (power) constraint~$P_i$ and capacity
  (energy) constraint~$E_i$; these are such that the store~$i$ may
  serve energy at any rate~$r_i(t)$ for each time~$t\in[0,T]$ subject
  to the constraints
  \begin{equation}
    \label{eq:1}
    0 \le r_i(t) \le P_i, \qquad t \in [0,T],
  \end{equation}
  and
  \begin{equation}
    \label{eq:2}
    \int_0^T r_i(t)\,dt \le E_i.
  \end{equation}
\end{compactenum}
Thus it is assumed in particular, from~\eqref{eq:2}, that stores may
not recharge energy during the time period~$[0,T]$.

Define also $r_{\S}(t)=\{r_i(t),\,i\in\S\}$ to be the set of rates at
which energy is served by the stores in~$\S$ at each time~$t\in[0,T]$.
We shall say that a \emph{policy} for the management of the set~$\S$
of stores over the time period $[0,T]$ is a process
$(r_{\S}(t),\,t\in[0,T])$ of such rates such that the
constraints~\eqref{eq:1} and~\eqref{eq:2} are satisfied and
additionally, without loss of generality, is such that
\begin{equation}
  \label{eq:3}
  \sum_{i\in\S} r_i(t) \le d(t), \qquad t\in[0,T].
\end{equation}
When the demand process $(d(t),\,t\in[0,T])$ is deterministic we
require that any policy $(r_{\S}(t),\,t\in[0,T])$ for the use of the
stores as above should also be deterministic.  When the demand process
$(d(t),\,t\in[0,T])$ is stochastic then we allow that, at each
time~$t$, the set of rates~$r_{\S}(t)$ may depend on the realised
value $(d(u),\,u\in[0,t])$ of the demand process to time~$t$.

Given any policy $(r_{\S}(t),\,t\in[0,T])$ for the management of the
stores as above, let the function~$w$ defined on the positive real
line be such that $w(d)$ is the cost per unit time associated with
a given level~$d$ of the \emph{residual demand process}
$(d(t) - \sum_{i\in\S} r_i(t),\,t\in[0,T])$.  Then the problem of
optimally scheduling the use of the stores typically becomes:
\begin{compactitem}
\item[\Pb:] choose a policy $(r_\S(t),\,t\in[0,T])$ to
  \emph{minimise}
  \begin{equation}
    \label{eq:4}
    \E\int_0^T w\biggl(d(t) - \sum_{i\in\S} \rt_i(t)\biggr)\,dt,
  \end{equation}
\end{compactitem}
where~$\E$ denotes the expectation operator.  (The latter is only
required when the demand process $(d(t),\,t\in[0,T])$ is stochastic.)
The case $w(d)=d$ corresponds to the minimisation of \emph{expected
  energy unserved} (EEU).  A complete solution to the optimal
scheduling problem in this case, valid in both deterministic and
stochastic environments, is given in
Section~\ref{sec:minimisation-eeu}.

However, it is often the case that the unit cost of unserved energy is
higher at higher levels of unmet demand; for example, modest levels of
unmet demand may often be dealt with by actions such as voltage
reduction, while, at the other extreme, the highest levels of unmet
demand may result in major blackouts and consequent societal
disruption.  Thus it may sometimes be natural to consider a
function~$w$ which, instead of being linear, is more generally
increasing and convex with $w(0)=0$.  Here the optimal
scheduling problem~\Pb\ corresponds to the minimisation of a form of
weighted EEU in which the marginal cost of unserved energy is an
increasing function of the level of residual demand.  The solution
to this problem is considerably more complicated and is discussed in
Section~\ref{sec:minim-weight-eeu-1}.

\section{Storage and demand profiles}
\label{sec:stor-demand-prof}

In this section we consider the question of whether any given demand
process $(d(t),\,t\in[0,T])$
may be completely satisfied by some given set~$\S$ of stores, i.e.\
whether there exists a policy $(r_\S(t),\,t\in[0,T])$ (satisfying the
constraints~\eqref{eq:1} and~\eqref{eq:2}) such that
\begin{equation}
  \label{eq:5}
  \sum_{i\in\S}r_i(t) = d(t), \qquad t \in [0,T].
\end{equation}
We give a simple, and readily testable, necessary and sufficient
condition for this to be possible.  This makes it easy to check ahead
of time whether a given set of stores is adequate to meeting a given
demand process, and, if necessary, to appropriately adjust that set.
Central to establishing the sufficiency of the condition is the use of
a particular policy for the prioritisation of the use of stores at any
time.  Let~$E_i(t)$ be the residual energy in each store $i\in\S$ at
each time~$t$; note that $E_i(0)=E_i$; define also the \emph{residual
  time} of each store~$i$ at each time~$t$ as $E_i(t)/P_t$ (this is
the length of further time for which the store~$i$ could supply energy
at its maximum rate).  Suppose that at time~$t$ it is desired to serve
energy at a total rate~$\hat d(t)$, where
$\hat d(t)\le\sum_{i\in\S\colon E_i(t)>0} P_i$ (the rate~$\hat d(t)$
might be the demand $d(t)$ to be met at time~$t$ or might be some
lesser rate); group the stores according to their residual
times~$E_i(t)/P_i$ at that time (so that two stores belong to the same
group if and only if their residual times are equal); rank the groups
in \emph{descending order} of their residual times, and select in this
order just sufficient groups of stores such that, using these stores
at their maximum rates (i.e.\ each selected store~$i$ is used at a
rate $r _t(t)=P_i$), the required total rate~$\hat d(t)$ is met; the
exception to this rule is that, in order to exactly meet the total
rate~$\hat d(t)$, the stores in the \emph{last} group thus selected
may only require to be used at some fraction~$\lambda$ of their
maximum rates, so that each store~$i$ in this last group is required
to supply energy at rate~$r_i(t)=\lambda P_i$ for some
common~$\lambda\le1$.  We shall refer to a policy in which, at each
time~$t$, the use of stores is prioritised as above as a \emph{longest
  residual time first} (\ltf) policy.  It is uniquely determined given
the total rate~$\hat d(t)$ at which energy is served at each
time~$t\in[0,T]$.  This policy is also considered in~\cite{EAT} which
shows that any process $(\hat d(t),\,t\in[0,T])$ which can be entirely
satisfied by a given set of stores can be satisfied by supplying
energy via the corresponding \ltf\ policy.  Thus in particular we may
check whether a given demand process $(d(t),\,t\in[0,T])$ may be thus
satisfied by simply checking numerically whether the \ltf\ policy
succeeds in doing so.  However, a simpler and more readily checkable
condition is desirable.

The effect of any \ltf\ policy, as time~$t$ progresses, is to
gradually equalise the residual lifetimes~$E_i(t)/P_i$ over stores;
further once the residual lifetimes of any set of stores have become
equal they remain so.  As also noted in~\cite{EAT}, this latter
property has the immediate consequence of preserving through time the
ordering of stores by their residual times, i.e.\ under an \ltf\
policy, for any pair of stores $i,j\in\S$ and for any time~$t$,
\begin{equation}
  \label{eq:6}
  E_i(t)/P_i\ge E_j(t)/P_j
  \quad\Rightarrow\quad
  E_i(t')/P_i\ge E_j(t')/P_j \quad\text{for all times~$t'>t$}.
\end{equation}
Suppose now that a given demand process $(d(t),\,t\in[0,T])$ is served
via the use of the unique \emph{greedy} \ltf\ policy, where by this it
is meant that the total rate $\sum_{i\in\S} \rt_i(t)$ at which demand
is served at each time~$t$ is given by
\begin{equation}
  \label{eq:7}
  \sum_{i\in\S} \rt_i(t) = \min\Biggl(d(t),\sum_{i\in\S:E_i(t)>0}P_i\Biggr),
\end{equation}
and this demand is then served using the corresponding \ltf\ policy.
(Note that this policy is available whether the demand process is
deterministic or stochastic, the latter since the rates~$\rt_{\S}(t)$
at each time~$t$ do not depend on the demand process subsequent to
time~$t$.)  Let~$\Se$ be the set of stores whose capacity constraints
are strictly binding under the greedy \ltf\ policy, i.e.\ the set of
stores which not only empty by time~$T$ but do so strictly prior to
the time $T'=\sup\{t\le T: d(t)>\sum_{i\in\S:E_i(t)>0}P_i\}$; define
also $\Sne=\S\setminus\Se$ to be the complementary set.
It is a immediate consequence
of the definition of the greedy \ltf\ policy---in particular
of~\eqref{eq:6} and~\eqref{eq:7}---that, again under this policy, the
use of the stores in the set $\Sne$ is prioritised over those in the
set $\Se$ throughout the \emph{entire} time period~$[0,T]$ so that,
for all~$t\in[0,T]$,
\begin{equation}
  \label{eq:8}
  \sum_{i\in\Sne} \rt_i(t) =  \min\left(d(t),\sum_{i\in\Sne} P_i\right).
\end{equation}

We now consider whether any given demand process (which we here treat
as if it were known in advance) may be totally satisfied by a given
set of stores.  In doing so it is clear that we may reorder the
succession of time instants within the period~$[0,T]$ so that this
demand process is (weakly) decreasing over time.  Thus, given any
demand process $(d(t),\,t\in[0,T])$, it is convenient to define its
corresponding \emph{demand profile} $(d^*(t),\,t\in[0,T])$ (often
referred to in the engineering literature as the \emph{load duration
  curve}, see, e.g.~\cite{BillAll}) as the given demand process
reordered as above, i.e.\ $(d^*(t),\,t\in[0,T])$ is the unique
nonincreasing nonnegative process on $[0,T]$ such that, for all $d'\ge0$,
\begin{equation}
  \label{eq:9}
  m(\{t: d^*(t) \le d'\}) = m(\{t: d(t) \le d'\}),
\end{equation}
where the function~$m$ applied to any set of times within the
interval~$[0,T]$ defines the total length of that set.%
\footnote{Formally, the function~$m$ denotes Lebesgue measure, and
  there is an assumption that any demand process is Lebesgue
  measurable.  Since, in applications, time is usually considered as a
  succession of discrete intervals on each of which everything is
  constant, there are no practical difficulties here.}

Similarly, given the set~$\S$ of stores as described in
Section~\ref{sec:model}, define its \emph{storage profile} as the
nonnegative function~$(s_{\S}(t),t\ge0)$ defined on the positive
half-line, such that, for all~$t\ge0$, $s_{\S}(t)$ is the rate at which
energy is supplied at time~$t$ when, starting at time~$0$, every store
is discharged continuously at its maximum rate, i.e.\
\begin{equation}
  \label{eq:10}
  s_{\S}(t) = \sum_{i\in\S\colon E_i/P_i \ge t} P_i.
\end{equation}
For example, in the case of a single store with capacity
constraint~$E$ and rate constraint~$P$, its storage profile is the
function~$s_{\S}$ given by $s_{\S}(t)=P$ for $t\in[0,E/P]$ and
$s_{\S}(t)=0$ for $t>E/P$.  Note that storage profiles are
\emph{additive}, i.e.\ the storage profile of the union of two
disjoint sets of stores is the sum of their individual storage
profiles.  In general, the storage profile~$(s_{\S}(t),t\ge0)$ of a
set~$\S$ of stores is (weakly) decreasing, and it is clear (and is
formally a consequence of Theorem~\ref{thm:nsc} below) that, in terms
of the model of the present paper, two sets of stores have equivalent
capabilities if and only if their storage profiles are the same.
Notably, in the case of a single store as above, a given demand
process $(d(t),\,t\in[0,T])$ may be completely served by that store if
and only if $d(t)\le P$ for all $t\in[0,T]$ and
$\int_0^T d(t)\,dt \le E$.  A demand process satisfying these
conditions may similarly be served by a set~$\S$ of stores such that
$E_i/P_i$ is the same for all~$i\in\S$, and $\sum_{i\in\S}P_i=P$ and
$\sum_{i\in\S}E_i=E$: it is only necessary to ensure that at every
time all stores are used at rates proportional to their maximal rates,
so that the residual lifetimes $E_i(t)/P_i$ are kept equal at all
times~$t$ and so no store empties before the final time~$T$.  The
following result generalises this observation and gives a simple
necessary and sufficient condition for a given demand process to be
capable of being satisfied by a given set of stores.

\begin{theorem}
  \label{thm:nsc}
  A given demand process $(d(t),\,t\in[0,T])$ may be completely
  satisfied by a given set~$\S$ of stores if and only if, for all
  $t\in[0,T]$,
  \begin{equation}
    \label{eq:11}
    \int_0^t s_{\S}(u)\,du \ge \int_0^t d^*(u)\,du,
  \end{equation}
  where, as defined above, $(d^*(t),\,t\in[0,T])$ is the demand
  profile (load duration curve) corresponding to the given demand
  process.
\end{theorem}

\begin{proof}
  As discussed above, the demand process $(d(t),\,t\in[0,T])$ may be
  satisfied by the set~$\S$ of stores if and only if its corresponding
  demand profile $(d^*(t),\,t\in[0,T])$ may be so satisfied.  For
  the latter the energy to be served in any time interval~$[0,t]$ is
  given by the integral on the right side of~\eqref{eq:11}.  That it is
  necessary that the condition~\eqref{eq:11} holds for all $t\in[0,T]$
  now follows from the observation that, for each such~$t$, the
  integral on the left side of~\eqref{eq:11} is the maximum energy
  which can be served by the stores in the time interval~$[0,t]$.

  To show sufficiency suppose now that the condition~\eqref{eq:11}
  holds for all $t\in[0,T]$.  Suppose also that that the original
  demand process $(d(t),\,t\in[0,T])$ is served (as far as possible)
  using the greedy \ltf\ policy.  
  It follows from the earlier observation~\eqref{eq:8} that, at each
  time~$t\in[0,T]$, the total rate at which energy is then served by
  the stores in the set~$\Sne$ defined above is given by
  $\min(d(t),\hat P)$, where $\hat P=\sum_{i\in\Sne}P_i$.  Thus the
  residual demand process $(d_e(t),\,t\in[0,T])$ to be served by the
  stores in the complementary set~$\Se$ is given by
  $d_e(t)=\max(0,d(t)-\hat P)$ for all~$t\in[0,T]$.  Since, for
  each~$t$, $d_e(t)$ is an increasing function of $d(t)$, the residual
  \emph{demand profile} corresponding to the residual demand process
  $(d_e(t),\,t\in[0,T])$ is similarly given by
  \begin{equation}
    \label{eq:12}
    d_e^*(t)=\max(0,d^*(t)-\hat P), \qquad t\in[0,T],
  \end{equation}
  while the storage profile defined by the stores in the residual
  set~$\Se$ is given by
  \begin{equation}
    \label{eq:13}
    s_{\Se}(t)=\max(0,s_{\S}(t)-\hat P), \qquad t \ge 0.
  \end{equation}
  Let~$T'=\sup\{t\colon d^*(t)\ge \hat P\}$; note that $T'\le T$.  Then
  \begin{align*}
    \int_0^T s_{\Se}(t)\,dt
    & \ge \int_0^{T'} s_{\Se}(t)\,dt\\
    & \ge \int_0^{T'} s_{\S}(t)\,dt - T' \hat P\\
    & \ge \int_0^{T'} d^*(t)\,dt - T' \hat P\\
    & = \int_0^T d_e^*(t)\,dt\\
    & = \int_0^T d_e(t)\,dt,
  \end{align*}
  where the second inequality above follows from~\eqref{eq:13}, the
  third inequality follows from the condition~\eqref{eq:11}, and the
  immediately succeeding equality follows from the
  condition~\eqref{eq:12}.  Thus the initial energy in the set of
  stores~$\Se$ is at least as great as the residual demand
  $\int_0^T d_e(t)\,dt$ to be met by those stores.  Since, by
  definition, the stores in the set~$\Se$ are empty at time~$T$, it
  follows that the residual demand process $(d_e(t),\,t\in[0,T])$
  is entirely served by those stores by time~$T$.
\end{proof}


\section{Minimisation of EEU}
\label{sec:minimisation-eeu}

Suppose now that the objective function~$w$ of the optimal scheduling
problem~\Pb\ is given by $w(d)=d$, so that the problem becomes that of
the optimal use of such storage as is available for the minimisation
of EEU.  In the case where the demand process $(d(t),\,t\in[0,T])$ is
deterministic, the optimisation problem~\Pb\ is then formally a linear
programme.  However, the inequality nature of the
constraints~\eqref{eq:1} and~\eqref{eq:2} ensures that this problem
has a particularly simple solution, which we discuss below.  Further
this solution continues to be optimal when the demand process
$(d(t),\,t\in[0,T])$ is stochastic.

In identifying optimal policies in either a deterministic or a
stochastic environment, it is sufficient to consider those which are
\emph{greedy}, i.e.\ those in which, as previously, at each successive
time~$t$, the total rate $\sum_{i\in\S} r_i(t)$ at which energy is
served by the stores in the set~$\S$ is given by~\eqref{eq:7}, so that
any residual demand at the time~$t$ is reduced as far as possible.
Essentially the reason for this is that, when the objective is the
minimisation of unserved energy, nothing is to be gained by
withholding for later possible use energy which could have been used
to reduce further residual demand at any time~$t$; further, when the
demand process $(d(t),\,t\in[0,T])$ is stochastic, there is the risk
that any such withheld energy may turn out to be not needed
subsequently.  It follows that, in the case of a \emph{single} store,
it is clearly always optimal to follow the policy of using the energy
of the store as quickly as possible without wasting it---see also
Edwards \emph{et al.}~\cite{ESDT} for further analysis of this case.

In the case of multiple stores Edwards \emph{et al.}~\cite{ESDT}
consider various heuristic greedy policies, but show that none of
these is in general optimal.  However, the above simple and obvious
result for a single store is a special case of the following more
general result for multiple stores.

\begin{theorem}
  \label{thm:min_eeu}
  In either a deterministic or a stochastic environment, the
  minimisation of EEU over the period~$[0,T]$ is achieved by the
  unique greedy \ltf\ policy.
\end{theorem}

\begin{proof}
  Recall from Section~\ref{sec:stor-demand-prof} that the unique
  greedy \ltf\ policy may be followed in either a deterministic or a
  stochastic environment.  As in the proof of Theorem~\ref{thm:nsc},
  it follows from~\eqref{eq:8} that at each time~$t\in[0,T]$, the
  total rate at which energy is served by the stores in the set~$\Sne$
  (which is defined in Section~\ref{sec:stor-demand-prof} and which is
  random when the demand process $(d(t),\,t\in[0,T])$ is stochastic)
  is then given by $\min(d(t),\sum_{i\in\Sne}P_i)$.  Hence, for every
  realisation of the demand process $(d(t),\,t\in[0,T])$, the stores
  in the set~$\Sne$ make the maximum contribution of which they are
  capable to the reduction of unserved energy over the period $[0,T]$.
  Again under the greedy \ltf\ policy, the stores in the complementary
  set~$\Se$ are all empty at time~$T$ and hence also make the maximum
  contribution of which they are capable to the reduction of unserved
  energy over the period $[0,T]$.  Thus collectively the stores in the
  set~$\S$ minimise EEU over the period~$[0,T]$.
\end{proof}

\begin{remark}
  In the case where the demand process $(d(t),\,t\in[0,T])$ is
  deterministic and hence assumed known, the greedy \ltf\ policy is
  typically not the only policy which minimises EEU.  For example, the
  greedy \ltf\ policy might be applied in reverse time to also
  minimise EEU.  Other policies may also be available in a
  deterministic environment---see the further discussion of
  Section~\ref{sec:minim-weight-eeu-1}.  However, such alternatives
  require the demand process to be known in advance and hence are
  typically not available when the latter is stochastic.
\end{remark}

\begin{remark}
  Observe that Theorem~\ref{thm:min_eeu} gives a very simple proof of
  the result of~\cite{EAT}, discussed in
  Section~\ref{sec:stor-demand-prof}, that any demand process
  $(d(t),\,t\in[0,T])$ which can be entirely satisfied by a given set
  of stores (so that the corresponding EEU is zero) can be satisfied
  by supplying energy via the corresponding \ltf\ policy.
\end{remark}

\section{Contributions of individual stores to EEU minimisation}
\label{sec:contr-indiv-stor}

In applications it is often necessary to determine the contribution
$\eeu(\S) - \eeu(\S\cup\{i\})$ to the overall reduction in EEU made by
the addition of any further store~$i$ to the set~$\S$.  This
contribution is readily determined, via, e.g.\ the use of the greedy
\ltf\ policy.  (For example, arguments entirely analogous to those
used in Section~\ref{sec:stor-demand-prof} show that if, under the
above policy, the additional store~$i$ is emptied over the time
interval $[0,T]$, then over that time interval the stores in the
original set~$\S$ contribute the same energy as previously, and so the
reduction in EEU achieved by the additional store~$i$ is
precisely~$E_i$.)

However, of particular interest in energy applications is the
determination of the \emph{equivalent firm capacity} (EFC) of any such
additional store~$i$ to be added to the set~$\S$.  This is now used,
for example in Great Britain, in determining the relative values of
different stores and other forms of capacity-providing resource within
electricity capacity markets---see~\cite{NGECR18}.  A formal
definition of EFC is given, for example, in~\cite{ZD2011}.  For
present purposes, the EFC~$\efc_i$ of any such store~$i$ may be
defined as that level of \emph{firm capacity} (ability to supply
energy at any rate up to a given constant) which, if added to the
set~$\S$ instead of the store~$i$, would achieve the same reduction in
minimised EEU.  We consider the case where the contributions of
individual stores to the overall reduction in unserved energy are
marginal, i.e.\ relatively small.  This is commonly the situation in
electricity capacity markets---see~\cite{Ofgem2018}.

In order to determine the EFC of storage resources, it is therefore
necessary to understand the change in minimised EEU which occurs when
the set of stores is supplemented by a small amount of firm capacity
as above.  Hence, for any $z\ge0$, let $\eeu(\S,z)$ denote the
minimised EEU achieved by the use of the set of stores~$\S$
supplemented by firm capacity equal to~$z$ (note that the latter makes
the same contribution as a store able to supply energy at rate~$z$ and
with very large or unlimited capacity).  The function $\eeu(\S,z)$ is
easily seen to be convex in~$z$; let $\eeu'(\S)$ denote its (right)
derivative with respect to~$z$ at $z=0$.  For any $\S'\subseteq\S$
define $\lole(\S')$ to be the \emph{loss-of-load expectation}
(see~\cite{BillAll}, and also \cite{DECCAnnexC} for a description of
the use of this metric in GB capacity adequacy analysis)~associated
with the minimised EEU achieved by (the sole use of) the set of
stores~$\S'$, \emph{when the stores in this set are scheduled via the
  use of the greedy \ltf\ policy}.  This is the expectation of the
\emph{loss-of-load} duration, i.e.\ the total length of time for which
the use of the set of stores~$\S'$ in accordance with the greedy \ltf\
policy fails to meet fully the demand process $(d(t),\,t\in[0,T])$.
We allow that the set~$\S'$ may be random, as is the case for the
set~$\Sne$ defined in Section~\ref{sec:stor-demand-prof} when the
demand process $(d(t),\,t\in[0,T])$ is stochastic.  Note also, from
the observation~\eqref{eq:8}, that the greedy \ltf\ policy schedules
the stores within the set~$\Sne$ in the same way whether this set is
considered on its own or as a subset of the entire set of stores~$\S$.
We now have the following result.

\begin{theorem}
  \label{thm:deriv}
  The derivative $\eeu'(\S)$, defined as above, is given by
  \begin{equation}
    \label{eq:14}
    \eeu'(\S) = - \lole(\Sne).
  \end{equation}
\end{theorem}

\begin{proof}
  Suppose that the set~$\S$ is supplemented by firm
  capacity~$\delta z>0$ where $\delta z$ is (infinitesimally) small.
  Under the greedy \ltf\ policy the latter may be considered to be
  scheduled as an additional store with rate~$\delta z$ and no
  capacity constraint.  Since $\delta z$ is small, under the above
  policy, the sets~$\Sne$ and $\Se$ remain otherwise unchanged and the
  stores in the set~$\Sne$ also continue to behave as if they were
  capacity unconstrained and their use continues to be completely
  prioritised over that of the stores in the set~$\Se$.  It follows
  from~\eqref{eq:8} that, in the absence of the stores in the
  set~$\Se$, the reduction in unserved energy caused by the addition
  of the firm capacity~$\delta z$ would be given by~$\delta z$ times
  the total length of time~$t$ within the period~$[0,T]$ such that
  $d(t)>\sum_{i\in\Sne}P_i$, i.e.\ by~$\delta z$ times the
  loss-of-load duration associated with the set~$\Sne$.  Thus, taking
  expectations and noting that no additional problems are caused by
  the fact that in a stochastic environment the set $\Sne$ is random,
  we have
  \begin{equation}
    \label{eq:15}
    \eeu(\Sne,\delta z) = \eeu(\Sne) - \delta z \times \lole(\Sne).
  \end{equation}
  Further, again under the greedy \ltf\ policy and under the addition
  of the firm capacity~$\delta z$, the stores in the remaining
  set~$\Se$ continue to contribute all their energy, so that,
  from~\eqref{eq:15},
  \begin{displaymath}
    \eeu(\S,\delta z) = \eeu(\S) - \delta z \times \lole(\Sne).
  \end{displaymath}
  The claimed result~\eqref{eq:14} now follows on dividing by
  $\delta z$ and letting $\delta z\to0$.
\end{proof}


\begin{remark}
  In an electricity system in which no storage is involved in
  balancing that system (i.e.\ the set $\S$ is empty), the conclusion
  of Theorem~\ref{thm:deriv} reduces to the known result that the
  derivative of any residual EEU with respect to variation by firm
  capacity is equal to the negative of the corresponding LOLE.  This
  result is sometimes used as the basis of an economic derivation of
  the current GB reliability standard---see~\cite{DECCAnnexC}.
\end{remark}

It follows from definitions of $\eeu'(S)$ and equivalent firm capacity
given above that the EFC~$\efc_i$ of any marginal (small) store~$i$
added to the set~$\S$ is given by the solution of
\begin{displaymath}
  \eeu(\S) - \eeu(\S\cup\{i\}) = \eeu'(\S) \times \efc_i.
\end{displaymath}
It now follows from this and from Theorem~\ref{thm:deriv} that this
EFC is given by
\begin{equation}
  \label{eq:16}
  \efc_i = \frac{\eeu(\S) - \eeu(\S\cup\{i\})}{\lole(\Sne)}.
\end{equation}
Both the numerator and the denominator in the expression on the right
side of~\eqref{eq:16} are readily determined in applications---though
in a random environment this will typically require simulation.


\begin{remark}
  Note that the EFC of any store is defined independently of the exact
  way in which the use of any set of stores is scheduled, provided
  that scheduling is such that EEU is minimised.  However, the
  quantity $\lole(\Sne)$ in equation~\eqref{eq:16} is defined in terms
  of the outcome of the greedy \ltf\ policy.  (Other optimal
  scheduling policies are possible and would result in a different
  value of this quantity.)
\end{remark}

Theorem~\ref{thm:deriv} and the consequent result~\eqref{eq:16} have
important consequences in understanding the contribution to be made by
further resource in the presence of existing storage.  As a simple
example, consider a single store---e.g.\ a large pumped storage
facility such as Dinorwig in Great Britain (see~\cite{Din})---which is
used in accordance with the (optimal) greedy policy to cover as far as
possible a single period of shortfall whose length (in the absence of
the store) is given by~$a$.  Thus the store is used at its maximum
rate, subject to energy not being wasted, until such time as the store
is emptied or the shortfall period is finished.  Let the total length
of the residual period of time within the above period during which
there is still strictly positive energy shortfall be given by~$b$,
where necessarily $0\le b\le a$.  Suppose now that a very small amount
of (further) firm capacity~$\delta z$ is made available for further
support of the system.  It follows from Theorem~\ref{thm:deriv} that
the further reduction in EEU achieved by the use of this additional
firm capacity is given by~$a\times\delta z$ in the case where the
capacity constraint of the store is strictly binding ($\S=\Se$---so
that in particular the store is empty at the end of the shortfall
period), and by~$b\times\delta z$ otherwise.  Essentially the reason
for this dichotomy is as follows: (in either case) under the greedy
\ltf\ policy the additional firm capacity makes an energy contribution
$a\times\delta z$ over the original shortfall period; in the case
where the capacity constraint of the store is strictly binding, when
the additional firm capacity is added the use of the store is
rescheduled so that it continues to empty; in the case where the
capacity constraint of the store is \emph{not} strictly binding, under
the additional firm capacity the store becomes less useful than
before.  It now further follows from~\eqref{eq:16} that if, instead of
the addition of further firm capacity, there is added instead a
marginal (i.e.\ small) further store which is able to be completely
utilised over the shortfall period, the EFC of that further store is
\emph{less} in the case where the capacity constraint on the original
store is strictly binding.

\section{Minimisation of weighted EEU}
\label{sec:minim-weight-eeu-1}

We now consider the case where the objective function~$w$ of the optimal
scheduling problem~\Pb\ is a more general convex increasing function
(with $w(0)=0$ as before), so that the problem~\Pb\ is that of the
minimisation of a form of weighted EEU.  As discussed in
Section~\ref{sec:model} this corresponds to the situation where the unit
economic cost of unserved energy is higher at higher levels of unmet
demand.  This is something which is often considered to be the case in
applications (see~\cite{NGECR18}), even if in the absence of storage it
is rarely formally modelled (presumably because there is then relatively
little opportunity to schedule generation so as to target particularly
those times at which demand is highest).

The optimal scheduling problem~\Pb\ is now a nonlinear constrained
optimisation problem.  In its solution it makes sense to concentrate
such storage resources as are available at the times of highest
demand.  Thus in general it is no longer the case that an optimal
policy will be greedy, since at times of low demand, storage resources
may be better withheld for later use at those times at which they may
be used more effectively.  The optimal decision at each successive
time~$t$ now generally depends, in a deterministic environment, on the
entire demand process $(d(t),\,t\in[0,T])$ and, in a stochastic
environment, on the entire distribution of this demand process.

We assume for the present that the demand process $(d(t),\,t\in[0,T])$
is deterministic and known in advance.  Under this assumption, we give
a complete solution of the optimal scheduling problem~\Pb.  It turns
out that this solution is independent of the form of the objective
function~$w$, subject only to its satisfying the given convexity
condition (something which is not true in a stochastic
environment---see Section~\ref{sec:examples-1}).  Thus the given
solution remains a solution when, as in
Section~\ref{sec:minimisation-eeu} the objective function~$w$ is
linear and given by $w(d)=d$ for all $d\ge0$, even though the solution
given here looks very different from the greedy solution which is
given in that section.  Note, however, that the latter is also optimal
in a stochastic environment.

For any $P>0$ and any $k\ge0$, define the functions~$f_{P,k}$ and
$\bar f_{P,k}$ on~$\Rp$ by
\begin{equation}
  \label{eq:17}
  f_{P,k}(d) =
  \begin{cases}
    0,     & \qquad d \le k,\\
    d - k, & \qquad k < d < k + d,\\
    P,     & \qquad d \ge k + P,\\
  \end{cases}
\end{equation}
and 
\begin{equation}
  \label{eq:18}
  \bar f_{P,k}(d) = d - f_{P,k}(d).
\end{equation}
Note that both these functions are nonnegative and (weakly) increasing.

Consider first the case where there is a single store with rate
constraint~$P$ and capacity constraint~$E$.  It is reasonably
clear---and it is formally a special case of
Theorem~\ref{thm:weighted-minimisation} below---that when, as at
present, the function~$w$ is convex the optimal solution to the
problem~\Pb\ is given by serving all demand in excess of some level~$k$,
subject to the rate constraint~$P$, i.e.\ by serving energy at rate
\begin{equation}
  \label{eq:19}
  \hat\rt(t)=f_{P,k}(d(t)), \qquad t \in [0,T].
\end{equation}
Here the constant~$k\ge0$ is chosen as small as possible subject to the
capacity constraint
\begin{equation}
  \label{eq:20}
  \int_0^T \hat\rt(t)\,dt \le E.
\end{equation}
The residual (unserved) demand at each time~$t\in[0,T]$ is then given by
$\bar f_{P,k}(d(t))$.

In the case~$k=0$ the optimality of the above solution is obvious; for
$k>0$ note that then the capacity constraint~\eqref{eq:20} (with the
process $(\hat\rt(t),t\in[0,T])$ defined by~\eqref{eq:19}) is
necessarily satisfied with equality; the process
$(\hat\rt(t),\,t\in[0,T])$ serves as much demand above the level~$k$ as
is possible, and it follows from the convexity of the function~$w$ that
any switching of energy resource from serving demand above the level~$k$
to serving demand below the level~$k$ cannot lead to an improved
solution to~\Pb.

Now consider the case where the set~$\S$ consists of multiple stores.
Here a solution to the problem~$\Pb$ is given by taking these stores in
any order and using each in succession in accordance with the above
solution for a single store.  We formalise this result in
Theorem~\ref{thm:weighted-minimisation} below and provide a formal
proof.

\begin{theorem}
  \label{thm:weighted-minimisation}
  Suppose that the set~$\S$ consists of stores~$i=1,\dots,n$ (where
  these stores may be taken in any order), and that, as usual, each
  store~$i$ has rate constraint~$P_i$ and capacity constraint~$E_i$.
  Then a solution to the optimal scheduling problem~$\Pb$ is given by
  using each store~$i$ in accordance with the rate process
  $(\hat\rt_i(t),t\in[0,T])$, where these rate processes are defined
  inductively in~$i$ by
  \begin{equation}
    \label{eq:21}
    \hat\rt_i(t)=f_{P_i,k_i}(d_{i-1}(t)), \qquad t \in [0,T],
  \end{equation}
  where $k_i\ge0$ is chosen as small as possible subject to the
  capacity constraint
  \begin{equation}
    \label{eq:22}
    \int_0^T \hat\rt_i(t)\,dt \le E_i,
  \end{equation}
  and where, for each $t\in[0,T]$, the quantity
  $d_{i-1}(t)=d(t)-\sum_{j=1}^{i-1}\hat\rt_j(t)$ is the residual
  demand at time~$t$ after the use of the stores $j=1,\dots,i-1$ (with
  $d_0(t)$ defined to be equal to $d(t)$ for all $t\in[0,T]$).
  Further the total rate process
  $(\sum_{i\in\S}\hat\rt_i(t),t\in[0,T])$ is independent of the order
  in which the stores within the set~$\S$ are taken in the above
  construction.
\end{theorem}

\begin{proof}
  For each $i=1,\dots,n$, the residual demand process
  $(d_i(t),t\in[0,T])$, corresponding to the demand remaining to be
  served after the use of the stores $j=1,\dots,i$ as in the statement
  of the theorem, is given by
  \begin{equation}
    \label{eq:23}
    d_i(t) = \bar f_{P_i,k_i}(d_{i-1}(t)), \quad t \in [0,T].
  \end{equation}
  Define also
  \begin{equation}
    \label{eq:24}
    k^*_i = \bar f_{P_n,k_n} \dots \bar f_{P_{i+1},k_{i+1}} (k_i),
    \quad i = 1,\dots,n-1,
    \qquad
    k^*_n = k_n,
  \end{equation}
  where $fg(k)=f(g(k))$ denotes functional composition.  (Note that,
  for any time~$t$ such that, under the scheduling policy described in
  the statement of the theorem, the residual demand after the use of
  the stores $j=1,\dots,i$ is given by $d_i(t)=k_i$, the corresponding
  residual demand after the use of all~$n$ stores is given by
  $d_n(t)=k^*_i$.)  Observe that, from~\eqref{eq:17}
  and~\eqref{eq:18}, for all $i=1,\dots,n$,
  \begin{equation}
    \label{eq:25}
    0 \le k^*_i \le k_i.
  \end{equation}
  For all~$i$, define $\lambda_i = w'(k^*_i)$, where in the event
  that the function $w$ fails to be differentiable at $k^*_i$ we take
  $w'(k^*_i)$ to be the left derivative (which exists by the
  convexity of the function~$w$) and where we take $w'(0)=0$.
  For all $t\in[0,T]$, define also $\bar\lambda(t) = w'(d_n(t))$.

  For any $t\in[0,T]$ and $i=1,\dots,n$, we make the following
  observations:
  \begin{compactenum}[(i)]
  \item if $d_i(t) > k_i$ then, since
    $d_i(t) = \bar f_{P_i,k_i}(d_{i-1}(t))$ and
    $\hat\rt_i(t) = f_{P_i,k_i}(d_{i-1}(t))$, it follows
    from~\eqref{eq:17} and~\eqref{eq:18} that necessarily
    $d_{i-1}(t) > k_i + P_i$ and so also $\hat\rt_i(t) = P_i$; further,
    by the monotonicity of the
    functions~$\bar f_{P_{i+1},k_{i+1}}, \dots ,\bar f_{P_n,k_n}$, it
    follows from~\eqref{eq:23} and~\eqref{eq:24} that $d_n(t) \ge k^*_i$
    and so also $\bar\lambda(t)\ge\lambda_i$;
  \item if $d_i(t)=k_i$ then, arguing as in~(i), we have
    $0 \le \hat\rt_i(t) \le P_i$; further, again arguing as in~(i),
    $d_n(t) = k^*_i$ and so also $\bar\lambda(t)=\lambda_i$;
  \item if $d_i(t)<k_i$ then, again arguing as in~(i), we have
    $\hat\rt_i(t) = 0$; further, again arguing as in~(i),
    $d_n(t) \le k^*_i$ and so also $\bar\lambda(t) \le \lambda_i$.
  \end{compactenum}
  It follows from the observations (i)--(iii) that, for $t\in[0,T]$ and
  $i=1,\dots,n$,
  \begin{equation}
    \label{eq:26}
    \bar\lambda(t) > \lambda_i \ \Rightarrow \ \hat\rt_i(t) = P_i,
    \qquad
    \bar\lambda(t) < \lambda_i \ \Rightarrow \ \hat\rt_i(t) = 0.
  \end{equation}
  Since, for each $t\in[0,T]$, the constant~$\bar\lambda(t)$ is the
  slope of a supporting hyperplane to the convex function~$w$ at
  $d_n(t) = d(t) - \sum_{i=1}^n\hat\rt_i(t)$, it follows that, for any
  set of rates~$\rt_i(t)$, $i=1,\dots,n$, satisfying the rate
  constraints~\eqref{eq:1},
  \begin{align}
    w\left(d(t) - \sum_{i=1}^n \rt_i(t)\right)
    & \ge w\left(d(t) - \sum_{i=1}^n \hat\rt_i(t)\right)
      + \bar\lambda(t)\sum_{i=1}^n\left(\hat\rt_i(t) - \rt_i(t)\right)
    \nonumber \\
    & \ge w\left(d(t) - \sum_{i=1}^n \hat\rt_i(t)\right)
      + \sum_{i=1}^n \lambda_i \left(\hat\rt_i(t) - \rt_i(t)\right),
    \label{eq:27}
  \end{align}
  where the second line in the above display follows from the
  observations~\eqref{eq:26}.

  Note further that, for any~$i$, from the definition of the
  rates~$\hat\rt_i(t)$ and from~\eqref{eq:25},
  \begin{equation}
    \label{eq:28}
    \int_0^T \hat\rt_i(t)\,dt < E_i
    \ \Rightarrow \ k_i = 0 
    \ \Rightarrow \ k^*_i = 0 
    \ \Rightarrow \ \lambda_i = 0. 
  \end{equation}
  The function $w$ is increasing and so $\lambda_i = w'(k^*_i)
  \ge 0$ for all $i$.  Hence it follows from~\eqref{eq:28} that, 
  for any set of rate functions $(\rt_i(t), t\in[0,T])$, $i=1,\dots,n$,
  satisfying the capacity constraints~\eqref{eq:2},
  \begin{equation}
    \label{eq:29}
    \lambda_i \int_0^T\left(\hat\rt_i(t) - \rt_i(t)\right) dt \ge 0,
    \qquad i = 1,\dots n.
  \end{equation}
  Hence, finally, for any set of rate functions $(\rt_i(t), t\in[0,T])$,
  $i=1,\dots,n$, which are \emph{feasible} for the problem~$\Pb$, i.e.\
  satisfy both the rate constraints~\eqref{eq:1} and capacity
  constraints~\eqref{eq:2}, it follows from~\eqref{eq:27}
  and~\eqref{eq:29} that
  \begin{displaymath}
    \int_0^T w\left(d(t) - \sum_{i=1}^n \rt_i(t)\right) dt
    \ge
    \int_0^T w\left(d(t) - \sum_{i=1}^n \hat\rt_i(t)\right) dt,
  \end{displaymath}
  so that the given rate functions $(\hat\rt_i(t), t\in[0,T])$,
  $i=1,\dots,n$, solve the problem~\Pb\ as required.

  To show that the total rate process
  $(\sum_{i\in\S}\hat\rt_i(t),t\in[0,T])$ is independent of the order
  in which the stores within the set~$\S$ are taken in the
  construction of the theorem, consider any strictly convex
  function~$w$ in the statement of the problem~\Pb.  The objective
  function~\eqref{eq:4} of that problem is then a strictly convex
  function of the total rate process
  $(\sum_{i\in\S}\rt_i(t),t\in[0,T])$ and so is then minimised at a
  unique value of that total rate process.  The latter, however, is
  given by $(\sum_{i\in\S}\hat\rt_i(t),t\in[0,T])$ and so this sum has
  the asserted independence property.
\end{proof}

\begin{remark}
  We observe that it is of course straightforward that the rates
  $(\hat\rt_i(t),t\in[0,T])$ at which the \emph{individual}
  stores~$i\in\S$ contribute in the construction of
  Theorem~\ref{thm:weighted-minimisation} in general do depend on the
  order in which the stores are taken within that construction.
\end{remark}

We now discuss briefly the case where the demand process
$(d(t),\,t\in[0,T])$ is stochastic.  Here the solution of the
problem~\Pb\ (when the ``weighting'' function~$w$ is nonlinear) is
complex and, as remarked above, depends on the probabilistic
distribution of the entire demand process.  Further, unlike in the
case where the demand process is deterministic, this solution is no
longer independent of the choice of the function~$w$---see
Example~\ref{ex:2} below.  Stochastic dynamic programming provides a
possible approach, but this requires that the distribution of the
demand process $(d(t),\,t\in[0,T])$ is sufficiently known.  In
practice it may be desirable to take a more heuristic approach.  For
example, at each time~$t\in[0,T]$ one might make a deterministic
estimate of the future demand process, and then make the optimal
decision, as described above, as to how much energy to serve at that
time on the basis of that estimate; a new estimate of the future
demand process might then be made at each subsequent time and the
amount of energy to be then served correspondingly re-optimised.  Such
\emph{deterministic re-optimisation} or \emph{rolling intrinsic}
techniques are robust and known to work well in stochastic
environments where the underlying probability distributions are
themselves uncertain---see, for example,~\cite{Sec2015} and the
references therein.

\section{Examples}
\label{sec:examples-1}

In this section we give two further examples which between them
illustrate much of the theory of the present paper.

\begin{example}\label{ex:1}
  We give a numerical example which, although simple, is nevertheless
  reasonably realistic in the case of the use of storage to cover what
  might otherwise be a significant period of shortfall in a country
  such as Great Britain.  We take $T=8$ half-hour periods (these being
  the time units used in Great Britain for the initial attempt in the
  balancing of supply and demand ahead of real time) and the demand
  process to be met from storage---after the exhaustion of all
  generation---to be given by $(400,400,400,400,1000,1000,200,200)$ MW
  on these successive half-hour time periods.  This might reasonably
  correspond to a credible period of fairly severe shortfall such as
  might occur during an early evening period of peak demand.
  (However, when as at present most electricity capacity is provided
  by generation, such a period of shortfall to be met by storage would
  be rare.)  Recall from Section~\ref{sec:stor-demand-prof} that any
  set of stores with identical values of $E_i/P_i$ are equivalent to a
  single larger store; a corollary of this is that, in choosing
  examples, there is no loss of generality for illustrative purposes
  in assuming all stores to have the same rate constraint.  Thus
  consider $5$ stores each with $P_i=200$ MW and with $E_i=500$ MWh,
  $400$ MWh, $400$ MWh, $300$ MWh, $200$ MWh for $i=1,2,3,4,5$.  These
  values might be those appropriate to a number of large batteries.
  It is easy to check that the greedy \ltf\ policy empties all the
  stores over the time period~$[0,6]$ (the first 3 hours of the
  shortfall period) and eliminates all the shortfall over that time
  period, serving a total of $1800$ MWh of energy and leaving
  unsatisfied all the shortfall during the period $[6,8]$, thus giving
  a (residual) minimised EEU of $200$ MWh.  However, this is not the
  only policy capable of minimising EEU.  For example, the greedy
  \ltf\ policy applied to the demand process in reverse time
  eliminates all shortfall over the time period $[1,8]$ and none of
  the shortfall occurring during the time period $[0,1]$; however,
  such a ``time-reversed'' policy could not be attempted in a
  stochastic environment (see Section~\ref{sec:minimisation-eeu}).
  Consider also the policy which uses the \ltf\ prioritisation of
  stores to serve as far as possible all shortfall in excess of some
  level~$k$.  It turns out that for $k=50$ MW all shortfall in excess
  of the level~$k$ may be thus served and that the given stores are
  then empty at the final time~$8$, so that this policy again
  minimises EEU.  Arguments analogous to those in
  Section~\ref{sec:minim-weight-eeu-1} show that this policy is also
  optimal when the objective is the minimisation of weighted EEU for
  any convex increasing objective function~$w$ (with $w(0)=0$) as
  described in that section.  (We remark that in general, in the case
  of multiple stores, the solution to the problem of minimising
  weighted EEU does not take such a simple form.)  However,
  implementation of a policy of this form again requires a knowledge
  of the demand process ahead of the times at which decisions need to
  be implemented, and so again cannot be directly implemented in a
  stochastic environment.  Observe also that, in this example and as
  discussed in the example of Section~\ref{sec:contr-indiv-stor}, the
  addition of further marginal firm capacity~$\delta z$ MW reduces EEU
  by a further full $4\times\delta z$ MWh (exactly as if the stores
  were not present), with the consequences discussed in
  Section~\ref{sec:contr-indiv-stor} for the EFC of any further
  storage.

  Finally consider the heuristic greedy policy suggested
  in~\cite{ESDT} in which stores are arranged in some order and, with
  respect to that order, earlier stores are completely prioritised
  over later ones.  It is again easy to check that neither the above
  policy in which the stores are arranged in descending order of
  capacity, nor that in which they are arranged in ascending order of
  capacity, succeeds in emptying all the stores---the remaining stored
  energy at time~$8$ under each of these two policies being
  respectively $100$ MWh and $200$ MWh.  Hence neither of these two
  policies succeeds in minimising EEU.
\end{example}

\begin{example}\label{ex:2}
  We give an very simple, essentially mathematical, example to
  illustrate the additional difficulties when, as in
  Section~\ref{sec:minim-weight-eeu-1}, the objective function~$w$ of
  the optimal scheduling problem~\Pb\ is nonlinear---corresponding to
  the objective being the minimisation of a form of weighted EEU---and
  when additionally the demand process $(d(t),\,t\in[0,T])$ is
  stochastic.  We assume appropriate physical units throughout.  We
  take this objective function to be given by $w(d) = d^p$ for some
  $p\ge1$.  We take $T = 2$ and let the demand process
  $(d(t),\,t\in[0,2])$ be given by $d(t) = 2$ for $t \in [0,1]$ and
  $d(t) = k$ for $t \in [1,2]$ where $k$ is a random variable which is
  uniformly distributed on $[0,4]$.  Finally we consider a single
  store with capacity $E = 2$ and rate constraint $P = 2$.

  In the case $p=1$ the optimisation problem is that of minimising EEU
  and here, by Theorem~\ref{thm:min_eeu}, the optimal solution is
  given by the greedy policy of serving all demand in the time period
  $[0,1]$, thus emptying the store at time~$1$ and so serving no
  demand thereafter.  That this solution is here unique follows since,
  if the store is not emptied by time~$1$, there is then a nonzero
  probability that it may not be possible to empty it by the final
  time~$T=2$.

  We now consider the case $p>1$, corresponding to the problem of
  minimising weighted EEU.  If the random variable~$k$ is replaced by
  its expected value of~$2$ then, as in
  Section~\ref{sec:minim-weight-eeu-1}, for all $p>1$ the optimal
  solution is given by serving energy at a constant rate~$1$
  throughout the entire time period $[0,2]$.  For the original
  stochastic problem, we assume that the constant value~$k$ of the
  demand process over the period $[1,2]$ is not known until time~$1$.
  Arguing as in Section~\ref{sec:minim-weight-eeu-1}, it is
  straightforward to see that the optimal policy is that of serving
  energy at some constant rate~$x\le2$ during the time period $[0,1]$
  and then at constant rate $\min(k,2-x)$ during the time period
  $[1,2]$.  For given~$x$ the objective function (weighted EEU) to be
  minimised is then given by
  \begin{equation}
    \label{eq:30}
    (2 - x)^p + \frac{1}{4}\int_0^4 \max\left(0,(k+x-2)^p\right)\,dk.
  \end{equation}
  The latter quantity is equal to $(2-x)^p+\frac{(x+2)^{p+1}}{4(p+1)}$
  and it is then routine to show that the optimal value of~$x$ is
  unique and is a monotonic decreasing function of~$p$---as might be
  expected---which tends to~$2$ as $p\to1$ (corresponding to the
  earlier case $p=1$ where a greedy policy minimises the risk of
  failing to empty the store by time~$2$) and which tends to~$0$ as
  $p\to\infty$ (corresponding to the very high penalty attached to
  high levels of residual demand when~$p$ is large).  Here the
  solution to the stochastic problem coincides with that given by the
  above deterministic approximation (in which the random variable~$k$
  is replaced by its expected value) only for that value of~$p$ such
  that the optimal solution to the stochastic problem is given by
  $x=1$.  This is approximately $p=1.79$.
\end{example}

\section*{Acknowledgements}
\label{sec:acknowledgements-1}

The authors are grateful to the Isaac Newton Institute for
Mathematical Sciences in Cambridge for their funding and hosting of a
number of most useful workshops to discuss this and other mathematical
problems arising in particular in the consideration of the management
of complex energy systems.  They are further grateful to National Grid
plc for additional discussion on the Great Britain electricity market,
to the Engineering and Physical Sciences Research Council---grant
EP/I017054/1---and to a grant from the Simons Foundation for the
support of the research programme under which the present research is
carried out.

\bibliography{storage_refs_all}
\bibliographystyle{plain}

\end{document}